\documentclass[a4paper,twoside]{article}
\usepackage{a4}
\usepackage{amssymb}
\usepackage{amsmath}
\usepackage{upref}
\usepackage[active]{srcltx}
\usepackage[colorlinks,citecolor=blue,linkcolor=blue]{hyperref}
\usepackage[dvipsnames]{color}
\allowdisplaybreaks[2] 
\newcount\minutes \newcount\hours
\hours=\time
\divide\hours 60
\minutes=\hours
\multiply\minutes -60
\advance\minutes \time
\newcommand{\klockan}{\the\hours:{\ifnum\minutes<10 0\fi}\the\minutes}
\newcommand{\tid}{\today\ \klockan}
\newcommand{\prtid}{\smash{\raise 10mm \hbox{\LaTeX ed \tid}}}
\renewcommand{\prtid}{}
\makeatletter
\def\sectionmark#1{} 
\def\subsectionmark#1{}
\newcommand{\sectnr}{\ifnum \c@secnumdepth >\z@
                 \thesection.\hskip 1em\relax \fi}
\def\@evenhead{\footnotesize\rm\thepage\hfil\leftmark\hfil\llap{\prtid}}
\def\@oddhead{\footnotesize\rm\rlap{\prtid}\hfil\rightmark\hfil\thepage}
\def\tableofcontents{\section*{Contents} 
 \@starttoc{toc}}
\makeatother
\makeatletter
\def\@biblabel#1{#1.}
\makeatother
\makeatletter
\let\Thebibliography=\thebibliography
\renewcommand{\thebibliography}[1]{\def\@mkboth##1##2{}\Thebibliography{#1}
\addcontentsline{toc}{section}{References}
\frenchspacing 
\setlength{\@topsep}{0pt}
\setlength{\itemsep}{0pt}%
\setlength{\parskip}{0pt plus 2pt}%
}
\makeatother
\makeatletter
\def\mdots@{\mathinner.\nonscript\!.%
 \ifx\next,.\else\ifx\next;.\else\ifx\next..\else
 \nonscript\!\mathinner.\fi\fi\fi}
\let\ldots\mdots@
\makeatother
\makeatletter
\let\Enumerate=\enumerate
\renewcommand{\enumerate}{\Enumerate%
\setlength{\@topsep}{0pt}
\setlength{\itemsep}{0pt}%
\setlength{\parskip}{0pt plus 1pt}%
\renewcommand{\theenumi}{\textup{(\alph{enumi})}}%
\renewcommand{\labelenumi}{\theenumi}%
}
\let\endEnumerate=\endenumerate
\renewcommand{\endenumerate}{\endEnumerate\unskip}
\makeatother
\makeatletter
\def\@seccntformat#1{\csname the#1\endcsname.\quad}
\makeatother
\newcommand{\authortitle}[2]{\author{#1}\title{#2}\markboth{#1}{#2}}
\newcommand{\auth}[2]{{#2. #1}}
\newcommand{\authvar}[3]{{#2. #1 #3}}
\newcommand{\art}[6]{{\sc #1, \rm #2, \it #3\/ \bf #4 \rm (#5), \mbox{#6}.}}
\newcommand{\artprep}[3]{{\sc #1, \rm #2, \it #3.}}
\newcommand{\artin}[3]{{\sc #1, \rm #2,  in #3.}}
\newcommand{\arttoappear}[3]{{\sc #1, \rm #2, to appear in \it #3}}
\newcommand{\book}[3]{{\sc #1, \it #2, \rm #3.}}
\newcommand{\AND}{{\rm and }}
\RequirePackage{amsthm}
\newtheoremstyle{descriptive}%
  {\topsep}   
  {\topsep}   
  {\rmfamily} 
  {}          
  {\bfseries} 
  {.}         
  { }         
  {}          
\newtheoremstyle{propositional}%
  {\topsep}   
  {\topsep}   
  {\itshape}  
  {}          
  {\bfseries} 
  {.}         
  { }         
  {}          
\theoremstyle{propositional}
\newtheorem{thm}{Theorem}[section]
\newtheorem{prop}[thm]{Proposition}
\newtheorem{lem}[thm]{Lemma}
\newtheorem{cor}[thm]{Corollary}
\theoremstyle{descriptive}
\newtheorem{deff}[thm]{Definition}
\newtheorem{example}[thm]{Example}
\newtheorem{remark}[thm]{Remark}
\makeatletter
\renewenvironment{proof}[1][\proofname]{\par
  \pushQED{\qed}%
  \normalfont
  \trivlist
  \item[\hskip\labelsep
        \itshape
    #1\@addpunct{.}]\ignorespaces
}{%
  \popQED\endtrivlist\@endpefalse
}
\makeatother
\newcommand{\setm}{\setminus}
\renewcommand{\emptyset}{\varnothing}
\def\vint{\mathop{\mathchoice%
          {\setbox0\hbox{$\displaystyle\intop$}\kern 0.22\wd0%
           \vcenter{\hrule width 0.6\wd0}\kern -0.82\wd0}%
          {\setbox0\hbox{$\textstyle\intop$}\kern 0.2\wd0%
           \vcenter{\hrule width 0.6\wd0}\kern -0.8\wd0}%
          {\setbox0\hbox{$\scriptstyle\intop$}\kern 0.2\wd0%
           \vcenter{\hrule width 0.6\wd0}\kern -0.8\wd0}%
          {\setbox0\hbox{$\scriptscriptstyle\intop$}\kern 0.2\wd0%
           \vcenter{\hrule width 0.6\wd0}\kern -0.8\wd0}}%
          \mathopen{}\int}
\newcommand{\Cp}{{C_p}}
\DeclareMathOperator{\diam}{diam}
\DeclareMathOperator{\capp}{cap}
\newcommand{\cp}{\capp_p}
\newcommand{\cone}{\capp_1}
\newcommand{\cpw}{\capp_{p,w}}
\newcommand{\cpbw}{\capp_{p,\bw}}
\newcommand{\conebw}{\capp_{1,\bw}}
\DeclareMathOperator{\dist}{dist}
\newcommand{\loc}{_{\rm loc}}
\newcommand{\simge}{\gtrsim}
\newcommand{\simle}{\lesssim}
{\catcode`p =12 \catcode`t =12 \gdef\eeaa#1pt{#1}}      
\def\accentadjtext#1{\setbox0\hbox{$#1$}\kern   
                \expandafter\eeaa\the\fontdimen1\textfont1 \ht0 }
\def\accentadjscript#1{\setbox0\hbox{$#1$}\kern 
                \expandafter\eeaa\the\fontdimen1\scriptfont1 \ht0 }
\def\accentadjscriptscript#1{\setbox0\hbox{$#1$}\kern   
                \expandafter\eeaa\the\fontdimen1\scriptscriptfont1 \ht0 }
\def\accentadjtextback#1{\setbox0\hbox{$#1$}\kern       
                -\expandafter\eeaa\the\fontdimen1\textfont1 \ht0 }
\def\accentadjscriptback#1{\setbox0\hbox{$#1$}\kern     
                -\expandafter\eeaa\the\fontdimen1\scriptfont1 \ht0 }
\def\accentadjscriptscriptback#1{\setbox0\hbox{$#1$}\kern 
                -\expandafter\eeaa\the\fontdimen1\scriptscriptfont1 \ht0 }
\def\itoverline#1{{\mathsurround0pt\mathchoice
        {\rlap{$\accentadjtext{\displaystyle #1}
                \accentadjtext{\vrule height1.593pt}
                \overline{\phantom{\displaystyle #1}
                \accentadjtextback{\displaystyle #1}}$}{#1}}
        {\rlap{$\accentadjtext{\textstyle #1}
                \accentadjtext{\vrule height1.593pt}
                \overline{\phantom{\textstyle #1}
                \accentadjtextback{\textstyle #1}}$}{#1}}
        {\rlap{$\accentadjscript{\scriptstyle #1}
                \accentadjscript{\vrule height1.593pt}
                \overline{\phantom{\scriptstyle #1}
                \accentadjscriptback{\scriptstyle #1}}$}{#1}}
        {\rlap{$\accentadjscriptscript{\scriptscriptstyle #1}
                \accentadjscriptscript{\vrule height1.593pt}
                \overline{\phantom{\scriptscriptstyle #1}
                \accentadjscriptscriptback{\scriptscriptstyle #1}}$}{#1}}}}
\newcommand{\al}{\alpha}
\newcommand{\alp}{\alpha}
\newcommand{\ga}{\gamma}
\newcommand{\dmu}{d\mu}
\newcommand{\de}{\delta}
\newcommand{\eps}{\varepsilon}
\newcommand{\la}{\lambda}
\newcommand{\Om}{\Omega}
\newcommand{\p}{{$p\mspace{1mu}$}}
\newcommand{\R}{\mathbf{R}}
\newcommand{\Sp}{\mathbf{S}}
\newcommand{\Z}{\mathbf{Z}}
\newcommand{\Ga}{\Gamma}
\newcommand{\bw}{\widetilde{w}}
\newcommand{\bmu}{\tilde{\mu}}

\newcommand{\limplus}{{\mathchoice{\vcenter{\hbox{$\scriptstyle +$}}}
  {\vcenter{\hbox{$\scriptstyle +$}}}
  {\vcenter{\hbox{$\scriptscriptstyle +$}}}
  {\vcenter{\hbox{$\scriptscriptstyle +$}}}
}}

\newcommand{\Np}{N^{1,p}}
\newcommand{\Nploc}{N^{1,p}\loc}
\newcommand{\Xhat}{\widehat{X}}

\makeatletter
\newcommand{\setcurrentlabel}[1]{\def\@currentlabel{#1}}
\makeatother
\newcounter{saveenumi}
\numberwithin{equation}{section}
\newcommand{\imp}{\mathchoice{\quad \Longrightarrow \quad}{\Rightarrow}
                {\Rightarrow}{\Rightarrow}}
\newenvironment{ack}{\medskip{\it Acknowledgement.}}{}

\begin{document}

\authortitle{Anders Bj\"orn, Jana Bj\"orn
    and Juha Lehrb\"ack}
{The annular decay property and capacity estimates for thin annuli}

\author{
Anders Bj\"orn \\
\it\small Department of Mathematics, Link\"oping University, \\
\it\small SE-581 83 Link\"oping, Sweden\/{\rm ;}
\it \small anders.bjorn@liu.se
\\
\\
Jana Bj\"orn \\
\it\small Department of Mathematics, Link\"oping University, \\
\it\small SE-581 83 Link\"oping, Sweden\/{\rm ;}
\it \small jana.bjorn@liu.se
\\
\\
Juha Lehrb\"ack \\
\it\small Department of Mathematics and Statistics, University of Jyv\"askyl\"a,\\
\it\small P.O. Box 35\/ \textup{(}MaD\/\textup{)}, FI-40014 University of Jyv\"askyl\"a, Finland\/{\rm ;}
\it \small juha.lehrback@jyu.fi
\\
}

\date{}

\maketitle

\noindent{\small
 {\bf Abstract}. 
We obtain upper and lower bounds for the nonlinear variational
capacity of thin annuli in weighted $\R^n$ and in
metric spaces,
primarily under the assumptions of an annular decay property
and a Poincar\'e inequality.
In particular, if the measure has the $1$-annular decay property at $x_0$ and
the metric space supports a pointwise $1$-Poincar\'e inequality at $x_0$,
then the upper and lower bounds are comparable and we get a two-sided
estimate for thin annuli centred at $x_0$, which generalizes the known
estimate for the usual variational capacity in unweighted $\R^n$.
Most of our estimates are sharp, which we show by supplying several key
counterexamples.
We also characterize the $1$-annular decay property.
}

\bigskip


\medskip


\section{Introduction}

We assume throughout the paper that $1 \le p<\infty$ 
and that $X=(X,d,\mu)$ is a metric space equipped
with a metric $d$ and a positive complete  Borel  measure $\mu$ 
such that $0<\mu(B)<\infty$ for all balls $B \subset X$.
We also let $x_0\in X$ be  a fixed but arbitrary
point and $B_r=B(x_0,r)=\{x : d(x,x_0)<r\}$.

In this paper, we continue the study of sharp estimates for
the variational capacity $\cp(B_r,B_R)$, 
which we started in
Bj\"orn--Bj\"orn--Lehrb\"ack~\cite{BBL}. 
Therein we concentrated on the
case $0<2r\le R$, while in the present work we are interested in the
case where the 
annulus $B_R\setminus B_r$ is thin,
that is, $0<\frac 1 2 R\le r < R$.

Assume for a moment that the measure $\mu$ is doubling and that
the space $X$ supports a \p-Poincar\'e inequality.
Then it is well known that 
$\cp(B_r,B_{2r})\simeq \mu(B_r)r^{-p}$ holds for all $0<r<\frac{1}{8} \diam X$.
If in addition the exponents $0<q\le q'<\infty$ are such that 
\begin{equation}\label{eq:measure lower-upper}
\Bigl(\frac rR\Bigr)^{q'} \simle \frac{\mu(B_r)}{\mu(B_R)}
\simle \Bigl(\frac rR\Bigr)^{q},
\quad \text{if }
0<r\le R<\diam X,
\end{equation}
then, by~\cite[Theorem~1.1]{BBL}, 
\begin{equation}   \label{eq-from-ringcap}
\cp(B_r,B_R)\simeq \begin{cases}
     \mu(B_r)r^{-p}, & \quad\text{if } p<q, \\
     \mu(B_R)R^{-p}, & \quad\text{if } p>q',
                           \end{cases}
\end{equation}
when $0<2r\le R<\frac{1}{4} \diam X$. 
However, when $r$ is close to $R$ these estimates are no longer valid;
in particular, typically $\cp(B_r,B_R)\to\infty$ when 
$r \to R$ and $p>1$
(see Section~\ref{sect-infty} for more on when this holds).
Moreover, the difference in the 
growth bounds in~\eqref{eq:measure lower-upper} does not play any role
when $r$ is close to $R$, and so it is obvious that other properties of the space
determine the capacities of thin annuli. 

In (unweighted) $\R^n$ the following equalities hold
for capacities of annuli for all $0<r<R<\infty$ (see e.g.~\cite[p.~35]{HeKiMa}):
\begin{equation*}
 \cp(B_r,B_R) = \begin{cases}
              C(n,p)|R^{(p-n)/(p-1)}-r^{(p-n)/(p-1)}|^{1-p}, 
    & \text{if } p\notin \{1,n\},\\
              C(n,p)\bigl(\log\frac{R}{r}\bigr)^{1-n}, & \text{if } p = n,\\
              C(n,p)r^{n-1}, & \text{if } p = 1.
                      \end{cases}
\end{equation*}
When $0 < \tfrac{1}{2} R\le r < R$, these yield  the estimate
\begin{equation}\label{eq:in-R^n}
 \cp(B_r,B_R) \simeq \Bigl(1-\frac r R\Bigr)^{1-p}\frac{m(B_R)}{R^p},
\end{equation}
where $m$ is the $n$-dimensional Lebesgue measure.

The main goal in this paper is to find general conditions for the space $X$ 
under which estimates similar to \eqref{eq:in-R^n} hold. 
One such condition is the following measure decay property, which 
will play a crucial role in our results.

\begin{deff} \label{def-annular-decay}
Let $\eta > 0$. The measure $\mu$  has the 
\emph{$\eta$-annular decay \textup{(}$\eta$-AD\/\textup{)} 
property at $x \in X$},  
if there is a constant $C$ such that for all radii $0<r<R$ we have
\begin{equation}\label{eq:annular decay}
 \mu(B(x,R)\setminus B(x,r))\leq C \Bigl(1-\frac r R\Bigr)^\eta\mu(B(x,R)).
\end{equation}

If there is a common constant $C$ such that
\eqref{eq:annular decay} holds for all $x\in X$ (and all radii $0<r<R$),
then
$\mu$ has the 
\emph{global $\eta$-AD property}.
\end{deff}

For most of the results in this paper it will
be enough to require a pointwise AD property at $x_0$,
often together with pointwise versions of
(reverse) doubling and Poincar\'e inequalities.
This resembles the situation in \cite{BBL},
where for capacity estimates for nonthin annuli, such as~\eqref{eq-from-ringcap}, 
it was enough to require doubling (and reverse-doubling)
and Poincar\'e inequalities to hold pointwise.

The global AD property
was introduced (under the name volume regularity property)
in Colding--Minicozzi~\cite[p.\ 125]{ColdingMiniCPAM98}
for manifolds 
and independently by Buckley~\cite{Buck},
who called it the annular decay property,
for general metric spaces.
A variant of the global AD property was already used
in David--Journ\'e--Semmes~\cite[p.\ 41]{DaJoSe85}.
Later, the global AD-property  
has been used by many other authors.
See e.g.\ Buckley~\cite{Buck} and Routin~\cite{Routin} 
for more information and applications of 
the global AD property.
We have not seen any considerations related to
the pointwise AD property in the literature.

If $X$ is a length space
and $\mu$ is globally doubling, 
then $\mu$ has the global $\eta$-AD property for some $\eta>0$,
see the proof of Lemma~3.3 in \cite{ColdingMiniCPAM98}.
Example~\ref{ex-starving-snake} shows that the length space assumption
cannot be dropped.

The AD property implies 
the following upper bound for the variational capacity.

\begin{prop}\label{prop-cap upper decay-intro}
Assume that $\mu$ has the $\eta$-AD property
at $x_0$.
Then
\begin{equation} \label{eq-cap upper decay-intro}
\cp(B_r,B_R)\simle \Bigl(1-\frac {r} R\Bigr)^{\eta-p}\frac{\mu(B_R)}{R^p},
\quad \text{if } 0<r<R.
\end{equation}
If $\mu$ has the global $\eta$-AD property,
then the implicit constant is independent of $x_0$.
\end{prop}

The proof of this result is quite simple, and it is perhaps 
more interesting that there are similar lower bounds and that
the estimate is sharp, as we show in 
Example~\ref{ex-Buckley-new}.
The sharpness is true even if one assumes that $\mu$ 
has the global $\eta$-AD property.

Lower bounds for capacities are in general
considerably more difficult to obtain than upper bounds.
Here we use relatively simple means to obtain lower bounds 
similar to the upper bounds,
so that we obtain two-sided estimates 
as in~\eqref{eq:in-R^n}. 
The key assumption is, as usual, some type of Poincar\'e inequality. 
When both the $1$-AD property and the
$1$-Poincar\'e inequality are available,
our upper and lower bounds coincide, and we obtain the
following generalization of~\eqref{eq:in-R^n}, which is our main result.

\begin{thm}\label{thm-the-nice-case-intro}
Assume 
that $X$ 
supports a global\/ $1$-Poincar\'e inequality
and that $\mu$ has the global\/ $1$-AD property.
Then
\begin{equation} \label{eq-thm-the-nice-case-intro}
\cp(B_r,B_R)\simeq \Bigl(1-\frac {r} R\Bigr)^{1-p}\frac{\mu(B_R)}{R^p},
\quad \text{if\/ } 0<\frac{R}{2}\le r<R \le \frac{\diam X}{3}.
\end{equation}
\end{thm}

As in Proposition~\ref{prop-cap upper decay-intro} it is actually
enough to require pointwise versions of the assumptions, but then
the result is a bit more complicated to formulate; see 
Theorem~\ref{thm-the-nice-case} for the exact statement.
Nevertheless, even with 
global assumptions the parameters are sharp, 
see Example~\ref{ex-weighted-bow-tie}.
A different type of
two-sided estimate for capacity is obtained in Theorem~\ref{thm-annular-PI}.

The $1$-AD property, 
which Buckley~\cite{Buck} calls ``strong annular decay'',
is essential in both the 
upper and lower bounds of Theorem~\ref{thm-the-nice-case-intro}.
The $1$-AD property is even locally the best possible
AD property, under very mild assumptions,
see Proposition~\ref{prop:decay}.
To further illustrate this useful
property, we establish several characterizations of the $1$-AD property
in Section~\ref{sect-char-1-annular}.

\begin{ack}
A.\ B.\ and J.\ B.\
were supported by the Swedish Research Council.
J.\ L.\ was supported by the Academy of Finland (grant no.\ 252108) and the 
V\"ais\"al\"a Foundation of the Finnish Academy of Science and Letters.
Part of this research was done during several visits
of J.\ L.\ 
to Link\"oping University in 2012--15,
and one visit of A.\ B.\ to the University of Jyv\"askyl\"a in 2015. 
\end{ack}

\section{Preliminaries}
\label{sect-prelim}

In this section we introduce the necessary background notation on metric
spaces and in particular on Sobolev spaces and capacities in metric spaces.
See the monographs Bj\"orn--Bj\"orn~\cite{BBbook} and
Heinonen--Koskela--Shanmugalingam--Tyson~\cite{HKSTbook}
for more extensive treatments of these topics, including
proofs of most of the results mentioned in this section.

A \emph{curve} is a continuous mapping from an interval,
and a \emph{rectifiable} curve is a curve with finite length.
We will only consider curves which are nonconstant, compact
and 
rectifiable, and thus each curve can 
be parameterized by its arc length $ds$. 
The metric space $X$ is a \emph{length space} 
if whenever $x,y\in X$ and $\eps>0$,
there is a curve between $x$ and $y$ with length less than $(1+\eps)d(x,y)$.

A property is said to hold for \emph{\p-almost every curve}
if it fails only for a curve family $\Ga$ with zero \p-modulus, 
i.e.\ there exists $0\le\rho\in L^p(X)$ such that 
$\int_\ga \rho\,ds=\infty$ for every curve $\ga\in\Ga$.
Following Heinonen--Kos\-ke\-la~\cite{HeKo98},
we introduce upper gradients 
as follows 
(they called them very weak gradients).

\begin{deff} \label{deff-ug}
A Borel function $g \colon X \to [0,\infty]$  is an \emph{upper gradient} 
of a function $f\colon X \to [-\infty,\infty]$
if for all  curves  
$\gamma \colon [0,l_{\gamma}] \to X$,
\begin{equation} \label{ug-cond}
        |f(\gamma(0)) - f(\gamma(l_{\gamma}))| \le \int_{\gamma} g\,ds,
\end{equation}
where the left-hand side is considered to be $\infty$ 
whenever at least one of the 
terms therein is infinite.
If $g\colon X \to [0,\infty]$ is measurable 
and \eqref{ug-cond} holds for \p-almost every curve,
then $g$ is a \emph{\p-weak upper gradient} of~$f$. 
\end{deff}

The \p-weak upper gradients were introduced in
Koskela--MacManus~\cite{KoMc}. It was also shown there
that if $g \in L^p(X)$ is a \p-weak upper gradient of $f$,
then one can find a sequence $\{g_j\}_{j=1}^\infty$
of upper gradients of $f$ such that $g_j \to g$ in $L^p(X)$.
If $f$ has an upper gradient in $L^p(X)$, then
it has an a.e.\ unique \emph{minimal \p-weak upper gradient} $g_f \in L^p(X)$
in the sense that for every \p-weak upper gradient $g \in L^p(X)$ of $f$ we have
$g_f \le g$ a.e., see Shan\-mu\-ga\-lin\-gam~\cite{Sh-harm}
and Haj\l asz~\cite{Haj03}. 
Following Shanmugalingam~\cite{Sh-rev}, 
we define a version of Sobolev spaces on the metric measure space $X$.

\begin{deff} \label{deff-Np}
For a measurable function $f\colon X\to[-\infty,\infty]$, let 
\[
        \|f\|_{\Np(X)} = \biggl( \int_X |f|^p \, \dmu 
                + \inf_g  \int_X g^p \, \dmu \biggr)^{1/p},
\]
where the infimum is taken over all upper gradients $g$ of $f$.
The \emph{Newtonian space} on $X$ is 
\[
        \Np (X) = \{f: \|f\|_{\Np(X)} <\infty \}.
\]
\end{deff}
\medskip

The quotient space $\Np(X)/{\sim}$, where  $f \sim h$ if and only if $\|f-h\|_{\Np(X)}=0$,
is a Banach space and a lattice, see Shan\-mu\-ga\-lin\-gam~\cite{Sh-rev}.
In this paper we assume that functions in $\Np(X)$ are defined everywhere,
not just up to an equivalence class in the corresponding function space.
This is needed for the definition of upper gradients to make sense.
If $f,h \in \Nploc(X)$, then $g_f=g_h$ a.e.\ in $\{x \in X : f(x)=h(x)\}$,
in particular $g_{\min\{f,c\}}=g_f \chi_{\{f < c\}}$ for $c \in \R$.

\begin{deff}
The \emph{Sobolev \p-capacity} of an arbitrary set $E\subset X$ is
\[
\Cp(E) = \inf_u\|u\|_{\Np(X)}^p,
\]
where the infimum is taken over all $u \in \Np(X)$ such that
$u\geq 1$ on $E$.
\end{deff}

The Sobolev capacity is countably subadditive and it
is the correct gauge 
for distinguishing between two Newtonian functions. 
If $u \in \Np(X)$, then $u \sim v$ if and only if they differ only
in a set of capacity zero.
Moreover, 
if $u,v \in \Np(X)$ and $u= v$ a.e., then $u \sim v$. 
This is the main reason why, unlike in the classical Euclidean setting, 
we do not need to 
require the functions admissible in the definition of capacity to be $1$ in a 
neighbourhood of $E$. 
In (weighted or unweighted) $\R^n$,
$\Cp$ is the usual Sobolev capacity and $\Np(\R^n)$ and $\Np(\Om)$ are the 
refined Sobolev spaces as in
Heinonen--Kilpel\"ainen--Martio~\cite[p.~96]{HeKiMa},
see Bj\"orn--Bj\"orn~\cite[Theorem~6.7\,(ix) and Appendix~A.2]{BBbook}.

\begin{deff}
The measure
$\mu$ is \emph{doubling at $x$} if there 
is a constant $C>0$ such that 
\begin{equation}\label{eq:doubling-x0}
  \mu(B(x,2r))\le C \mu(B(x,r))
\quad \text{whenever }r>0.
\end{equation}
If \eqref{eq:doubling-x0} holds with the same constant $C>0$
for all $x \in X$, we say that $\mu$ is 
\emph{\textup{(}globally\/\textup{)} doubling}.

We also say that the measure $\mu$ is \emph{reverse-doubling at $x$}, if there
are constants $\ga,\tau>1$ such that 
\[ 
  \mu(B(x,\tau r))\ge \ga \mu(B(x,r))
  \quad \text{for all } 0<r\le \diam X/2\tau,
\] 
and that the measure $\mu$ is \emph{Ahlfors $Q$-regular} if 
$\mu(B(x,r)) \simeq r^Q$ for all $x\in X$ and all $r>0$.
\end{deff}

The global doubling condition is often assumed in the metric space literature, 
but for many of 
our estimates it will be enough to assume that $\mu$ is doubling at $x$.
If $X$ is connected, or more generally uniformly perfect
(see Heinonen~\cite{heinonen}),
and $\mu$ is globally doubling, then
$\mu$ is also reverse-doubling at every point, with uniform constants.
In the connected case,  one can choose $\tau>1$ arbitrarily and find $\ga>1$ 
independent of $x$, see e.g.\ Corollary~3.8 in~\cite{BBbook}.
If $\mu$ is merely doubling at $x$, then
the reverse-doubling at $x$ does not follow automatically
and has to be imposed separately whenever needed.

The $\eta$-AD property at $x_0$ easily implies that $\mu$ is doubling
at $x_0$. 
The converse is not true even if $X$ is a length space,
as seen by considering $m+\de_1$ on $\R$, where $\de_1$
is the Dirac measure at $1$, 
which is doubling at $0$, but
does not have the $\eta$-AD property at $0$ for any $\eta>0$.
(For an absolutely continuous example,
consider $\R$ equipped with the measure $w\,dx$, where
 $w(x)=\max\{1,1/|x-1|(\log|x-1|)^2\}$.)

\begin{deff} \label{def-PI}
We say that $X$ supports a \emph{\p-Poincar\'e inequality at $x$} if
there exist constants $C>0$ and $\lambda \ge 1$
such that for all balls $B=B(x,r)$, 
all integrable functions $f$ on $X$, and all (\p-weak)
upper gradients $g$ of $f$, 
\[ 
        \vint_{B} |f-f_B| \,\dmu
        \le C r \biggl( \vint_{\lambda B} g^{p} \,\dmu \biggr)^{1/p},
\] 
where $ f_B 
 :=\vint_B f \,\dmu 
:= \int_B f\, d\mu/\mu(B)$.
If $C$ and $\lambda$ are independent of $x$, we say that
$X$ supports a \emph{\textup{(}global\/\textup{)} \p-Poincar\'e inequality}.
\end{deff}

A nonnegative function  $w$ on $\R^n$
 is a \emph{\p-admissible weight} if
$d\mu:=w\,dx$ is globally doubling and $\R^n$ equipped
with $\mu$ supports a global \p-Poincar\'e inequality.
See Corollary~20.9 in \cite{HeKiMa}
(which is only in the second edition) and
Proposition~A.17 in \cite{BBbook}
for why this is equivalent to other definitions in the literature.

It is well known that if $X$ supports a global 
\p-Poincar\'e inequality, then $X$ is connected, but in fact
even a pointwise \p-Poincar\'e inequality is sufficient for this.

\begin{prop} \label{prop-connected}
If $X$ supports a \p-Poincar\'e inequality at $x_0$,
then $X$ is connected and 
$\Cp(S_R)>0$ for
every sphere $S_R=\{x:d(x,x_0)=R\}$ with $R<\diam X/2$.
\end{prop}

In particular, if $\mu$ is 
globally doubling and $X$ supports a global Poincar\'e inequality,
then $\mu$ is reverse-doubling and $\tau>1$ can be chosen arbitrarily. 

\begin{proof}
The first part is shown in the same way as in Proposition~4.2
in \cite{BBbook}.
For the second part, assume that $\Cp(S_R)=0$.
Then $0$ is a \p-weak upper gradient of $\chi_{B_R}$,
as \p-almost no curve intersects $S_R$, see
\cite[Proposition~1.48]{BBbook}.
Thus the \p-Poincar\'e inequality is violated for $B_{2R}$.
\end{proof}

\begin{deff}
Let $\Om\subset X$ be open. The \emph{variational 
\p-capacity} of $E\subset \Om$ with respect to $\Om$ is
\[
\cp(E,\Om) = \inf_u\int_{\Om} g_u^p\, d\mu,
\]
where the infimum is taken over all $u \in \Np(X)$
such that
$\chi_E \le u \le 1$ in $E$ and $u=0$ on $X \setm E$;
we call such functions $u$ \emph{admissible} for
$\cp(E,\Om)$. 
\end{deff}

Also the 
variational capacity 
is countably subadditive and coincides with the usual variational
capacity 
(see Bj\"orn--Bj\"orn~\cite[Theorem~5.1]{BBvarcap} for a proof
valid in weighted $\R^n$). 

Throughout the paper, we write $a \simle b$ if there is an implicit
 constant $C>0$ such that $a \le Cb$, where $C$ is independent of the 
essential parameters involved. We also write $a \simge b$ if $b \simle a$,
and $a \simeq b$ if $a \simle b \simle a$.

Recall that $x_0\in X$ is  a fixed but arbitrary
point and $B_r=B(x_0,r)$.

\section{Upper bounds for capacity}
\label{sect-upper bounds}

In this section we prove Proposition~\ref{prop-cap upper decay-intro}
and show its sharpness.

\begin{lem}  \label{lem-upper-simple}
If\/ $0 < r <  R $, then
\[ 
\cp(B_r,B_R)\le \frac{\mu(B_R\setm B_r)}{(R-r)^p}.
\] 
\end{lem}

\begin{proof}
The function
\[u(x)=\biggl(1-\frac{\dist(x,B_r)}{R-r}\biggr)_\limplus\] 
is admissible  for 
$\cp(B_r,B_{R})$, and $g=(R-r)^{-1}\chi_{B_R\setminus B_r}$
is an upper gradient of $u$.
We thus obtain that
\[
\cp(B_r,B_{R})  \le \int_{B_R} g^p\,d\mu = \frac{\mu(B_{R}\setminus B_r)}{(R-r)^{p}}.
 \qedhere
\]
\end{proof}

\begin{proof}[Proof of Proposition~\ref{prop-cap upper decay-intro}]
Using the $\eta$-AD property at $x_0$ and Lemma~\ref{lem-upper-simple},
we obtain that
\begin{align*}
\cp(B_r,B_{R}) & \le  \frac{\mu(B_{R}\setminus B_r)}{(R-r)^{p}} 
   \simle  \biggl(\frac {R-r} R\biggr)^\eta \frac{\mu(B_{R})}{(R-r)^{p}} 
     =  \biggl(\frac {R-r} R\biggr)^{\eta-p}\frac{\mu(B_{R})}{R^p}. \qedhere
\end{align*}
\end{proof}

\begin{remark} \label{rmk-upper-bound}
Note that if $\mu$ has no AD property (in which case we could 
say that $\mu$ has the ``$0$''-AD property), then 
Lemma~\ref{lem-upper-simple}
still gives
\[
\cp(B_r,B_{R})\le \Bigl(1-\frac r R\Bigr)^{-p}\frac{\mu(B_{R})}{R^{p}}.
\]
This is sharp by Example~\ref{ex-Buckley-new} below.

Moreover, if $\mu$ has local $\eta$-AD at $x_0$, in the sense that
there is some $R_0>0$ such that \eqref{eq:annular decay}
 holds for all $0<r<R<R_0$ as in 
Proposition~\ref{prop:decay} below, then \eqref{eq-cap upper decay-intro} holds
whenever $0<r<R<R_0$.
Similar local versions hold also for our other results.

It follows directly from the proof that the 
constant $C$ from Definition~\ref{def-annular-decay}
can be used as the 
implicit constant
in \eqref{eq-cap upper decay-intro}.
\end{remark}

The following example shows that 
Proposition~\ref{prop-cap upper decay-intro}
is sharp.

\begin{example} \label{ex-Buckley-new}
(This example 
has been
inspired by Example~1.3 in Buckley~\cite{Buck}.)
Let $x_0=0$, $0 < \eta < 1$ and $d \mu=w \, dx$ on $\R^n$, $n \ge 1$, where 
$w(x)=w(|x|)$ and 
\[
   w(\rho)=\max\{1,|\rho-1|^{\eta -1}\}.
\]
This is a Muckenhoupt $A_1$-weight, by
Theorem~II.3.4 in Garc\'\i a-Cuerva--Rubio de Francia~\cite{garcia-cuerva}, and
it is thus $1$-admissible, by Theorem~4 in J.~Bj\"orn~\cite{JB-Fenn}.
It is easily verified that $\mu(B_R)\simeq R^n$ for all $R>0$.
We also see that $\mu(B_1 \setm B_r) \simeq (1-r)^\eta$,
if $\frac{1}{2}\le r \le 1$.
One can check that this is the extreme case showing that the measure $\mu$
has the  global $\eta$-AD property (and that $\eta$ is optimal).
By Proposition~10.8 in Bj\"orn--Bj\"orn--Lehrb\"ack~\cite{BBL},
for $p>1$ and $\tfrac{1}{2} \le r < 1$,
\begin{align}
    \cpw(B_r,B_1) 
   &\simeq      \biggl(\int_r^1 (w(\rho) \rho^{n-1})^{1/(1-p)}  \, d\rho \biggr)^{1-p} 
\label{eq-ringcap-est}
\\
   & \simeq \biggl( \int_r^1 (1-\rho)^{(\eta-1)/(1-p)}  \, d\rho \biggr)^{1-p}
   \simeq (1-r)^{\eta-p}, \notag
\end{align}
which shows that the upper bound in 
Proposition~\ref{prop-cap upper decay-intro}
is sharp, with $R=1$ fixed and $p>1$.

Now let $d\bmu = \bw\,dx$ and $d\mu_j=w_j\,dx$, where 
\[
\bw(\rho):=\sum_{j=0}^\infty a_j w_j(\rho), 
\quad \text{with }
w_j(\rho):= w(q_j \rho), \quad j=0,1,\ldots,
\]
for some countable set $\{q_j\}_{j=0}^\infty\subset (0,\infty)$
(e.g.\ all positive rational numbers)
and $a_j>0$ 
such that $\sum_{j=0}^\infty a_j<\infty$.
A change of variables shows that 
$\mu_j(B_R)=q_j^{-n} \mu(B_{q_j R}) \simeq R^n$ and hence
$\bmu(B_R)\simeq R^n$.
Moreover, for $0<r\le R$ and $x \in X$,
\begin{align*}
\bmu(B(x,R)\setm B(x,r)) &= \sum_{j=0}^\infty a_j \mu_j(B(x,R)\setm B(x,r)) \\
   &= \sum_{j=0}^\infty a_j q_j^{-n} \mu(B(q_jx,q_j R)\setm B(q_jx,q_j r)) \\
   &\simle \sum_{j=0}^\infty a_j q_j^{-n} \Bigl(1-\frac{r}{R}\Bigr)^\eta 
      \mu(B(q_jx,q_j R)) \\
   &=  \Bigl(1-\frac{r}{R}\Bigr)^\eta \bmu(B(x,R)),
\end{align*}
i.e.\ $\bmu$ has the global $\eta$-AD property as well.
Since $\bw\ge a_j w_j$ for every $j=0,1,\ldots$\,, we also see that $\eta$
is optimal.
Similarly, for every ball $B(x,r)\subset\R^n$, as $w$ is an $A_1$-weight,
\[
\vint_{B(x,r)} w_j\,dx = \vint_{B(q_j x,q_j r)} w\,dx \simle \inf_{B(q_j x,q_j r)} w
= \inf_{B(x,r)} w_j
\]
and summing over all $j$ shows that $\bw$ is an $A_1$-weight.
Finally, using Proposition~10.8 in~\cite{BBL} again
together with \eqref{eq-ringcap-est}, we obtain for $p>1$ and
$\tfrac12 q_j^{-1}\le r <R=q_j^{-1}$, 
\begin{align*}
\cpbw(B_r,B_R) 
   &\simeq q_j^{1-n} \biggl( \int_r^R \biggl( 
   \sum_{k=0}^\infty a_{k}  w(q_{k}\rho) 
        \biggr)^{1/(1-p)} \, d\rho \biggr)^{1-p} \\
   & \simge a_j q_j^{1-n} \biggl( q_j^{-1} \int_{q_j r}^1 w(\rho)^{1/(1-p)} 
        \, d\rho \biggr)^{1-p} \\
   &\simeq a_j q_j^{p-n}  \cpw(B_{q_j r},B_1) \\
   &\simeq a_j \Bigl(1-\frac{r}{R}\Bigr)^{\eta-p}  \frac{\bmu(B_R)}{R^p},
\end{align*}
and letting $r\nearrow R$ shows that the upper bound in 
Proposition~\ref{prop-cap upper decay-intro} is sharp for all $R=q_j^{-1}$.

For $p=1$  we cannot use Proposition~10.8 in~\cite{BBL}.
Instead we do as follows.
Let $\tfrac12 q_j^{-1}\le r <R=q_j^{-1}$, where $j=0,1,\ldots$\,.
Let $u$ be admissible for 
$\conebw(B_r,B_R)$,
and let $g$ be an upper gradient of $u$.
We then get, using the unweighted capacity $\cone(B_r,B_R)$ 
and~\eqref{eq:in-R^n}, 
\begin{align*}
\int_{\R^n} g \, d\bmu 
   & \ge \int_{\Sp^{n-1}} \int_r^R g (\rho \omega) \bw(\rho) \rho^{n-1}
      \,d\rho \, d\omega \\
   &\simge a_j w_j(r)  \int_{\Sp^{n-1}} \int_r^R g (\rho \omega) 
      \rho^{n-1} \, d\rho \, d\omega \\
   &\ge a_j (1-q_j r)^{\eta-1} \cone(B_r,B_R) \\
   & \simeq a_j \Bigl(1-\frac{r}{R}\Bigr)^{\eta -1} \frac{\mu(B_R)}{R}.
\end{align*}
Taking infimum over all admissible $u$ shows that 
Proposition~\ref{prop-cap upper decay-intro}
is sharp also for $p=1$.
\end{example}

We have the following observation showing
that the exponent $\eta=1$ is the largest that
can occur in the AD property, even locally, under a very mild assumption.

\begin{prop} \label{prop:decay}
Let $x_0\in X$ and $R_0 >0$, and assume that 
$\mu(\{x_0\})=0$.
If~\eqref{eq:annular decay} holds for some $\eta>0$ and
all\/ $0<r<R<R_0$, then $\eta\le 1$.
\end{prop}

\begin{proof}
Let $0<R<R_0$. 
Using~\eqref{eq:annular decay} 
we obtain for all integers $1\le k\le K$,
\begin{align}
 \mu(B_R)  &=\sum_{i=1}^{K}\mu(B_{iR/K}\setminus B_{(i-1)/RK})
     \le C \sum_{i=1}^{K} \biggl(1-\frac {i-1}{i} \biggr)^\eta\mu(B_{iR/K})     
  \notag \\
&\le C \sum_{i=1}^{k} {i}^{-\eta} \mu(B_{kR/K}) 
      + C \sum_{i=k+1}^\infty {i}^{-\eta} \mu(B_R),
      \label{eq:many annuli-new}
\end{align}
where $C$ is the constant in \eqref{eq:annular decay}.
If $\eta>1$, the series $\sum_{i=1}^\infty {i}^{-\eta}$ converges and
we can find $k$ such that $C \sum_{i=k+1}^\infty {i}^{-\eta}\le\tfrac12$.
Thus, subtracting the last term in \eqref{eq:many annuli-new} from the
left-hand side yields
\[
\mu(B_R) \simle \mu(B_{kR/K}) \to0, \quad \text{as }K\to\infty,
\]
which is impossible. 
Thus $\eta\le1$.
\end{proof}

\section{Lower bounds for capacity}
\label{sect:lower}

We now turn to lower estimates for capacities of thin annuli. The following
is our main estimate for obtaining the lower bound
in Theorem~\ref{thm-the-nice-case-intro}.
As usual for lower bounds,
 a key assumption is some sort of a  Poincar\'e inequality.

\begin{thm}\label{thm-PI-and-AD}
Assume that\/ $1\le q<p<\infty$, that 
$X$ supports a $q$-Poincar\'e inequality at $x_0$,
and that 
$\mu$ has the $\eta$-AD property
at $x_0$ and
is reverse-doubling at $x_0$ with dilation $\tau>1$.
Then
\begin{equation} \label{eq-PI-and-AD}
\cp(B_r,B_R)\simge \Bigl(1-\frac {r} R\Bigr)^{\eta(q-p)/q}\frac{\mu(B_R)}{R^p},
\quad  \text{if\/ }0<\frac{R}{2} \le r<R \le \frac{\diam X}{2\tau}.
\end{equation}
\end{thm}

If $\mu$ has the global $\eta$-AD property
and supports a global $q$-Poincar\'e inequality,
then the implicit constants are independent of $x_0$
and we may choose $\tau>1$ arbitrarily, see the proof of 
Theorem~\ref{thm-the-nice-case-intro} below.

Example~\ref{ex-weighted-1D} shows that the reverse-doubling
assumption cannot be dropped,
while  Example~\ref{ex-weighted-bow-tie} shows 
that it is not enough to assume that $X$ supports
global $q$-Poincar\'e inequalities for all $q>p$.
Moreover, Example~\ref{ex-starving-snake} shows that
the $\eta$-AD property cannot be replaced by the assumption
that $\mu$ is globally doubling or even Ahlfors regular.

\begin{proof}
Let $u$ be admissible for $\cp(B_r,B_R)$.
Then $u=1$ in $B_r$, $u=0$ outside $B_R$, 
and $g_u=0$ a.e.\ 
outside $B_R\setminus B_r$. 
The reverse-doubling implies that 
$\mu(B_{\tau R} \setm B_R) \simge \mu(B_R)$ 
from which it follows that 
$|u_{B_{\tau R}}|<c < 1$, and so $|u-u_{B_{\tau R}}| >  1-c>0$ in $B_{R/2}$.
Note that the $\eta$-AD property implies that $\mu$ is doubling at $x_0$.
Thus we obtain from the $q$-Poincar\'e inequality at $x_0$
and H\"older's inequality that
\begin{align*}
1 & \simle\vint_{B_{R/2}}|u-u_{B_{\tau R}}|\,d\mu 
   \simle \vint_{B_{\tau R}}|u-u_{B_{\tau R}}|\,d\mu
\simle R\biggl(\vint_{B_{\tau\lambda  R}} g_{u}^q \,d\mu \biggr)^{1/q}\\
& \simle \frac{R}{\mu(B_R)^{1/q}}\biggl(\int_{B_{R}\setminus B_r} g_{u}^q \,d\mu \biggr)^{1/q} \\
& \simle \frac{R}{\mu(B_R)^{1/q}} \mu(B_{R}\setminus B_r)^{1/q-1/p}\biggl(\int_{B_{R}\setminus B_r} g_{u}^p \,d\mu \biggr)^{1/p}.
\end{align*} 
By the $\eta$-AD property,
$\mu(B_{R}\setminus B_r)\simle (1-r/ R)^{\eta}\mu(B_R)$.
Inserting this into the above estimate yields 
\begin{align*}
\biggl(\int_{B_{R}\setm B_r} g_{u}^p \,d\mu \biggr)^{1/p} 
& \simge \frac {\mu(B_R)^{1/q}} R \Bigl(1-\frac {r} R\Bigr)^{\eta(1/p-1/q)}\mu(B_R)^{1/p-1/q}\\
& = \frac {\mu(B_R)^{1/p}} R \Bigl(1-\frac {r} R\Bigr)^{\eta(q-p)/pq},
\end{align*} 
and \eqref{eq-PI-and-AD} follows after taking infimum over all
admissible $u$.
\end{proof}

Theorem~\ref{thm-PI-and-AD} establishes the lower
bound in Theorem~\ref{thm-the-nice-case-intro} when $p>1$.
For $p=1$ we instead use the following result.
In view of Remark~\ref{rmk-upper-bound}, we can see
this as an $\eta=0$ version of Theorem~\ref{thm-PI-and-AD}.

\begin{prop}\label{prop-PI-lower-p=1}
Assume that $X$ supports a \p-Poincar\'e inequality at $x_0$,
and that $\mu$ is doubling at $x_0$ and reverse-doubling at $x_0$ 
with dilation $\tau>1$.
Then
\begin{equation} \label{eq-PI-and-AD-p=1}
\cp(B_r,B_R)\simge \frac{\mu(B_R)}{R^p},
\quad \text{if\/ }0<\frac{R}{2}\le r<R \le \frac{\diam X}{2\tau}.
\end{equation}
Moreover, 
$\cp(B_r,B_{2r})\simeq \mu(B_r)r^{-p}$ when  $0<  r \le \diam X/4\tau$.

If $\mu$ is globally doubling and 
$X$ supports a global \p-Poincar\'e inequality,
then the implicit constants are independent of $x_0$.
\end{prop}

Example~\ref{ex-starving-snake} shows that the
lower estimate is sharp even under the assumptions
that $\mu$ is globally doubling (or Ahlfors regular) and 
$X$ supports a global $1$-Poincar\'e inequality.
Example~\ref{ex-weighted-bow-tie} shows that
the \p-Poincar\'e assumption cannot be weakened, even if
it is assumed globally.
Example~\ref{ex-weighted-1D-inc} shows that the doubling
assumption cannot be dropped
(not even if $X$ supports a global $1$-Poincar\'e inequality
and $\mu$ is globally reverse-doubling),
while Example~\ref{ex-weighted-1D} shows that the reverse-doubling
assumption cannot be dropped.
See also Proposition~\ref{prop-PI-lower-p=1-no-doubl}.

\begin{proof}
Let $u$ be admissible for $\cp(B_r,B_R)$.
As in the proof of Theorem~\ref{thm-PI-and-AD}
(with $q$ replaced by $p$), we get that
\begin{align*}
1 &   \simle R \biggl(\vint_{B_{\tau\lambda  R}} g_{u}^p\,d\mu\biggr)^{1/p} 
 \simle \frac{R}{\mu(B_R)^{1/p}}
   \biggl(\int_{B_{R}\setminus B_r} g_{u}^p \,d\mu\biggr)^{1/p},
\end{align*} 
and~\eqref{eq-PI-and-AD-p=1} follows after taking infimum over all
admissible $u$.
That $\cp(B_r,B_{2r})\simeq \mu(B_r)r^{-p}$ follows from this
and Lemma~\ref{lem-upper-simple}.
\end{proof}

\begin{thm}\label{thm-the-nice-case}
Assume 
that
$X$ supports a\/ $1$-Poincar\'e inequality at $x_0$
and that $\mu$ has the\/ $1$-AD property
at $x_0$ and 
is reverse-doubling at $x_0$ with dilation $\tau>1$.
Then
\[ 
\cp(B_r,B_R)\simeq \Bigl(1-\frac {r} R\Bigr)^{1-p}\frac{\mu(B_R)}{R^p},
\quad 
\text{if\/ } 0<\frac{R}{2} \le r<R \le \frac{\diam X}{2\tau}.
\] 
\end{thm}

Even under global assumptions, as in Theorem~\ref{thm-the-nice-case-intro},
the $1$-Poincar\'e and $1$-AD assumptions cannot be weakened,
as shown by Example~\ref{ex-weighted-bow-tie}.
Example~\ref{ex-weighted-1D} shows that the reverse-doubling
assumption cannot be dropped, and in particular that
it is possible that $\mu$ has the
$1$-AD property
at $x_0$ and $X$ supports a $1$-Poincar\'e inequality at $x_0$,
but that $\mu$ fails to be reverse-doubling at $x_0$.

\begin{proof}
This follows by combining 
Proposition~\ref{prop-cap upper decay-intro}
with Theorem~\ref{thm-PI-and-AD}
(for $p>1$) and Proposition~\ref{prop-PI-lower-p=1} (for $p=1$).
\end{proof}

\begin{proof}[Proof of Theorem~\ref{thm-the-nice-case-intro}]
It follows from the global assumptions and Proposition~\ref{prop-connected}
that  $X$ is connected. 
Hence, $X$ is reverse-doubling at $x_0$ with $\tau=\frac{3}{2}$.
As the implicit constants in Theorem~\ref{thm-the-nice-case} only
depend on the parameters in the assumptions,
Theorem~\ref{thm-the-nice-case-intro} follows.
\end{proof}

The following result gives a two-sided estimate of a different form. 

\begin{thm}  \label{thm-annular-PI}
Assume that $\mu$ is globally doubling
and that $X$ supports a global \p-Poincar\'e inequality.
Let\/ $0 < \tfrac{1}{2}R \le r <  R $ and $\de=R-r$.
Assume, in addition, that there exists $a>0$ such that
for every $x\in B_R\setm B_r$ there exist 
$x'$ and $x''$ so that $B(x',a\de)\subset B(x,2\de)\cap B_r$ and 
$B(x'',a\de)\subset B(x,2\de)\setm B_R$.
Then
\[ 
\cp(B_r,B_R)\simeq \frac{\mu(B_R\setm B_r)}{(R-r)^p}.
\] 
\end{thm}

The balls $B(x',a\de)$ in the assumptions of 
Theorem~\ref{thm-annular-PI} always exist e.g.\ if $X$ is a length space.
The existence of the balls $B(x'',a\de)$ is more difficult to guarantee but
there are plenty of spaces where it is true. For example, $\R^n$ equipped 
with any \p-admissible measure satisfies the assumptions. 

Observe that the geometric condition is only assumed to hold for the specific
$r$ and $R$ under consideration, but whenever the geometric condition is
satisfied the implicit constants in the estimate are independent of $r$ and $R$.
Clearly, the constant 2 in $B(x,2\de)$ is not important and can be replaced
by any number $\ge 2$. 
This may be useful for some spaces containing well distributed holes.
The same is true for Corollary~\ref{cor-Jana-2} below.

That the geometric assumption cannot be dropped is shown by 
Example~\ref{ex-starving-snake}, while
Example~\ref{ex-weighted-bow-tie} shows 
that the Poincar\'e assumption cannot be 
weakened.
Examples~\ref{ex-weighted-1D} and~\ref{ex-weighted-1D-inc}
show that the assumption that $\mu$ is globally doubling can neither be replaced by 
the assumption that $\mu$ is doubling at $x_0$, nor by the assumption
that $\mu$ is globally reverse-doubling, i.e.\ reverse doubling at every
$x \in X$ with uniform constants.

\begin{proof}
The upper bound follows from Lemma~\ref{lem-upper-simple}, 
so it suffices to prove the lower bound.

Use the Hausdorff maximality principle to find a maximal pairwise disjoint 
collection of balls $B(x_j,\de)$ with $x_j\in B_R\setm B_r$.
By maximality, the balls $B_j=B(x_j,2\de)$ cover $B_R\setm B_r$.
Moreover, since $\mu$ is globally
doubling, it can be shown that the balls $\la B_j$
have bounded overlap depending only on $\la$ and the doubling constant of $\mu$.
Now, for each $j$ let $B'_j=B(x_j',a\de)$ and $B''_j=B(x_j'',a\de)$ 
as in  the assumption of the theorem.

Let $u$ be admissible for 
$\cp(B_r,B_R)$.
Then  $u=1$ in $B_r$ and $u=0$ outside $B_R$.
In particular, $u=1$ in each $B'_j$ and $u=0$ in each $B''_j$.
Since $\mu$ is globally doubling, it follows that 
$u_{B_j} \le 1-\mu(B''_j)/\mu(B_j)\le c$, where $c<1$ is independent of $j$.
An application of the global \p-Poincar\'e inequality to $B_j$,
and using that $g_u=0$ a.e.\ outside $B_R \setm B_r$, then yields

\begin{align*}
0<1-c & \le \vint_{B'_j} |u-u_{B_j}| \,d\mu \simle \vint_{B_j} |u-u_{B_j}| \,d\mu 
  \\
   & \simle \de \biggl(\frac{1}{\mu(\la B_j)} \int_{\la B_j} g_u^p \,d\mu \biggr)^{1/p} 
   =  \de \biggl(\frac{1}{\mu(\la B_j)} \int_{\la B_j\cap(B_R\setm B_r)} 
    g_u^p \,d\mu \biggr)^{1/p}. \notag
\end{align*} 
From this it follows that 
\[ 
\mu(\la B_j\cap(B_R\setm B_r)) \le
\mu(\la B_j) \simle \de^p \int_{\la B_j\cap(B_R\setm B_r)} g_u^p \,d\mu.
\] 
Since the balls $\la B_j$ have bounded overlap and
cover $B_R \setm B_r$, summing over all $j$ gives
\[
\mu(B_R\setm B_r) \simle \de^p \int_{B_R\setm B_r} g_u^p \,d\mu
\]
and taking infimum over all admissible $u$ 
proves the lower bound. 
\end{proof}

The following corollary partly complements Theorem~\ref{thm-PI-and-AD}
and the lower bound in Theorem~\ref{thm-the-nice-case}
in the case when the $1$-AD property or
a $1$-Poincar\'e inequality are not satisfied.
In particular, if the doubling condition 
and a $1$-Poincar\'e inequality hold globally and $p>1$, then the $1$-AD
condition can be replaced by the geometric condition in 
Theorem~\ref{thm-annular-PI}.
That the geometric assumption cannot be dropped is shown by 
Example~\ref{ex-starving-snake},
while 
Example~\ref{ex-weighted-bow-tie} shows 
that the pointwise Poincar\'e assumption cannot be weakened.

\begin{cor} \label{cor-Jana-2}
If the assumptions of Theorem~\ref{thm-annular-PI} are satisfied
and in addition $X$ supports a $q$-Poincar\'e inequality at $x_0$ for some\/
$1 \le q <p$,
 then
\[
\cp(B_r,B_R) \simge
 \Bigl(1-\frac {r}R \Bigr)^{q-p}\frac{\mu(B_R)}{R^p}.
\]
\end{cor}

\begin{proof}
This follows directly from Theorem~\ref{thm-annular-PI}
and the following Lemma~\ref{lem-reverse-decay-global}. 
\end{proof}

The following estimate complements the $1$-AD property. 
In particular, if $q=1$ then 
this lower bound, together with the $1$-AD property, 
leads to a sharp two-sided estimate
for $\mu(B_{R}\setm B_r)$ when $r$ is close to $R$.

\begin{lem} \label{lem-reverse-decay-global}
Assume that $X$ supports a $q$-Poincar\'e inequality at $x_0$
for some\/ $1 \le q < \infty$ and that 
$\mu$ is doubling at $x_0$ and
reverse-doubling at $x_0$.
Then 
\[ 
\mu(B_R\setminus B_r)\simge \Bigl(1-\frac r R\Bigr)^q\mu(B_R)
\quad \text{when\/ }0<\frac{R}{2}\le r < R< \frac{\diam X}{2\tau}.
\] 
\end{lem}

Example~\ref{ex-weighted-bow-tie} shows that
the Poincar\'e assumption cannot be weakened,
while Example~\ref{ex-weighted-1D} shows that
the reverse-doubling condition cannot be omitted.
We do not know if the doubling condition can be omitted.

\begin{proof}
Let
\[
   u(x)=\biggl(1-\frac{\dist(x,B_r)}{R-r}\biggr)_\limplus.
\]
As in the proof of Lemma~\ref{lem-upper-simple},
we obtain that
\[
\int_{X} g_{u}^q \,d\mu  \le \frac{\mu(B_{R}\setminus B_r)}{(R-r)^{q}}. 
\]
On the other hand, 
as in the proof of Theorem~\ref{thm-PI-and-AD},
we get that 
\begin{align*}
1 \simle R\biggl(\vint_{B_{\tau\lambda  R}} g_{u}^q \,d\mu \biggr)^{1/q}
\simle \frac{R}{\mu(B_R)^{1/q}} \frac{\mu(B_R\setminus B_r)^{1/q}}{(R-r)},
\end{align*} 
and the claim follows.
\end{proof}

We can also obtain the following variant of Corollary~\ref{cor-Jana-2}.

\begin{prop} \label{prop-Jana-3}
If the assumptions of Theorem~\ref{thm-annular-PI} are satisfied
with $p>1$, then there is\/ $1 \le q<p$ such that 
\[
\cp(B_r,B_R) \simge
 \Bigl(1-\frac {r}R \Bigr)^{q-p}\frac{\mu(B_R)}{R^p}.
\]
\end{prop}

As seen from the proof below (and those in \cite{AikSh05} and \cite{KeZh})
$q$ only depends on $p$ and the constants in the global doubling
condition and the global \p-Poincar\'e inequality.
Moreover, it also follows from the proof that the completion
$\Xhat$ of $X$ supports a global $q$-Poincar\'e inequality for this $q$.
In fact, it would be enough to require that
$\Xhat$ supports a $q$-Poincar\'e inequality at $x_0$.
Note that Koskela~\cite[Theorems~A and~C]{Koskela} 
has given counterexamples showing
that $X$ may not support any better Poincar\'e inequality
than the \p-Poincar\'e inequality
(so Corollary~\ref{cor-Jana-2} is not at our disposal).
His examples are of the type $X=\R^n \setm E$, where $E \subset \R^{n-1}$
so they satisfy the geometric condition in 
Theorem~\ref{thm-annular-PI} and $\Xhat$ supports a 
global $1$-Poincar\'e
inequality.

\begin{proof}
Let $\Xhat$ be the completion of $X$ and extend the measure
$\mu$ so that $\mu(\Xhat \setm X)=0$.
Then $\mu$ is doubling on $\Xhat$ and $\Xhat$ supports a \p-Poincar\'e
inequality, by Proposition~7.1 in Aikawa--Shanmugalingam~\cite{AikSh05}.
By Theorem~1.0.1 in Keith--Zhong~\cite{KeZh}, it 
follows that there is $1\le q<p$ such that $\Xhat$ supports a
$q$-Poincar\'e inequality.
Now we can apply 
Lemma~\ref{lem-reverse-decay-global} with respect to $\Xhat$, 
and since the estimate of Lemma~\ref{lem-reverse-decay-global} 
holds for the measure $\mu$ also when restricted to $X$, this estimate
together with Theorem~\ref{thm-annular-PI}
completes the proof.
\end{proof}

\section{The blowup of \texorpdfstring{$\cp(B_r,B_R)$}{cap(Br,BR)}
as \texorpdfstring{$r \to R$}{r-->R}}
\label{sect-infty}

Theorem~\ref{thm-PI-and-AD}, Corollary~\ref{cor-Jana-2} and
Proposition~\ref{prop-Jana-3} all give uniform
estimates for the blowup of $\cp(B_r,B_R)$ as $r \to R$.
In particular they show that 
\[
      \lim_{\de \to 0\limplus} \cp(B_{R-\de}, B_R) =
      \lim_{\de \to 0\limplus} \cp(B_R, B_{R+\de}) = \infty
\]
when the respective assumptions are satisfied.

If we are not interested in uniform estimates, but only in the limits
above, then these can be obtained under considerably weaker assumptions,
as we will now show.

\begin{prop} \label{prop-infty}
Assume that\/ $1\le q<p<\infty$ and that 
$X$ supports a $q$-Poincar\'e inequality at $x_0$.
Let $R>0$ be such that $\mu(X \setm B_R)>0$,
which in particular holds if $X \setm \itoverline{B}_R \ne \emptyset$.
Then
\begin{equation} \label{eq-infty}
\lim_{\de \to 0 \limplus}  \cp(B_{R-\de},B_R)=\infty.
\end{equation}
If in addition $\mu(\{y:d(y,x_0)=R\})=0$,
then also 
\begin{equation} \label{eq-infty-outer}
\lim_{\de \to 0 \limplus}  \cp(B_{R},B_{R+\de})=\infty.
\end{equation}
\end{prop}

If $X=B_R$ then $\cp(B_{R-\de},B_R)=\cp(B_{R},B_{R+\de})=0$, and thus the condition
$\mu(X \setm B_R)>0$ cannot be dropped for either limit.
Example~\ref{ex-weighted-bow-tie} shows 
that it is not enough to assume that $X$ supports
global $q$-Poincar\'e inequalities for all $q>p$ for neither
limit,
but we do not know if it is enough to assume
that $X$ supports a \p-Poincar\'e inequality
at $x_0$.
Moreover, Example~\ref{ex-starving-snake}  shows that
the assumption $\mu(\{y : d(y,x_0)=R\})=0$ cannot be dropped
for the limit \eqref{eq-infty-outer} 
to hold, even if 
$X$ supports a global $1$-Poincar\'e inequality.
If $p=1$ the result fails even if we assume
a global $1$-Poincar\'e inequality, as seen by considering
$\R^n$ or Theorem~\ref{thm-the-nice-case-intro}.

\begin{proof}
Assume that $0 <\de < \frac{1}{2}R$
and let $r=R-\de$.
Let $u$ be admissible for $\cp(B_r,B_R)$.
Then,
following the ideas in
the proof of  Theorem~\ref{thm-PI-and-AD},
\begin{align*}
 1 - \frac{\mu(B_R)}{\mu(B_{2R})} 
& \le \vint_{B_{R/2}}|u-u_{B_{2 R}}|\,d\mu \\
&   \le \frac{\mu(B_{2R})}{\mu(B_{R/2})} \vint_{B_{2 R}}|u-u_{B_{2R}}|\,d\mu \\ 
& \le CR \frac{\mu(B_{2R})}{\mu(B_{R/2})}
        \biggl(\vint_{B_{2\lambda  R}} g_{u}^q \,d\mu \biggr)^{1/q}\\
& =   \frac{CR\mu(B_{2R})}{\mu(B_{R/2})\mu(B_{2\la R})^{1/q}}
     \biggl(\int_{B_{R}\setminus B_r} g_{u}^q \,d\mu \biggr)^{1/q} \\
& \le \frac{CR\mu(B_{2R})}{\mu(B_{R/2})\mu(B_{2\la R})^{1/q}}
    \mu(B_{R}\setminus B_r)^{1/q-1/p}
    \biggl(\int_{B_{R}\setminus B_r} g_{u}^p \,d\mu \biggr)^{1/p}.
\end{align*} 
Taking infimum over all
admissible $u$ 
shows that
\begin{equation} \label{eq-infty-2}
  \cp(B_r,B_R) 
    \ge \biggl(1 - \frac{\mu(B_R)}{\mu(B_{2R})}\biggr)^p
    \biggl(\frac{\mu(B_{R/2})\mu(B_{2\la R})^{1/q}}{CR\mu(B_{2R})}\biggr)^p
    \mu(B_{R}\setminus B_r)^{1-p/q}.
\end{equation}
To see that the first factor in the right-hand side is positive, 
we note that 
either $X=\{y:d(y,x_0) \le R\}$
or
there is a point $y$ with $R < d(y,x_0)< \tfrac{3}{2} R$,
as $X$ is connected by Proposition~\ref{prop-connected}.
In the former case, $\mu(B_{2R} \setm B_R)=\mu(X \setm B_R)>0$,
while in the letter case 
$\mu(B_{2R} \setm B_R) \ge \mu(B(y,d(y,x_0)-R)) > 0$.
Thus the first factor in \eqref{eq-infty-2} is positive,
and so is clearly the second one as well.
Since the last factor tends to $\infty$, as $\de \to 0\limplus$,
we see that \eqref{eq-infty} holds.

The proof of \eqref{eq-infty-outer} is similar
(one can also use \eqref{eq-infty-2} directly), but
in this case one needs to use that 
$\mu(B_{R+\de} \setm B_R) \to \mu(\{y:d(y,x_0)=R\})=0$, as $\de \to 0\limplus$.
\end{proof}

\section{Characterizations of the 1-AD property}
\label{sect-char-1-annular}

Our aim in this section is to characterize the $1$-AD property.

\begin{thm} \label{thm-1-AD-global}
Let $f(r):=\mu(B_r)$.
Then the following are equivalent\/\textup{:}
\begin{enumerate}
\item \label{g-annular}
$\mu$ has the\/ $1$-AD property at $x_0$\textup{;}
\item \label{g-ac}
$f$ is locally absolutely continuous on\/ $(0,\infty)$ and 
$\rho f'(\rho) \simle f(\rho)$ for a.e.\ $\rho>0$\textup{;}
\item \label{g-lip}
\setcounter{saveenumi}{\value{enumi}}
$f$ is locally Lipschitz on\/ $(0,\infty)$ and 
$\rho f'(\rho) \simle f(\rho)$ for a.e.\ $\rho>0$.
\end{enumerate}

If moreover
$\mu$ is 
reverse-doubling at $x_0$ and $X$ supports a\/ $1$-Poincar\'e inequality at $x_0$,
then also
the following condition is equivalent to those  above\/\textup{:}
\begin{enumerate}
\setcounter{enumi}{\value{saveenumi}}
\item \label{g-extra}
$f$ is locally Lipschitz on\/ $(0,\infty)$ and 
$\rho f'(\rho) \simeq f(\rho)$ for a.e.\ $\rho$ with\/ 
$0< \rho < \diam X/2\tau$.
\end{enumerate}
\end{thm}

The assumption of absolute continuity cannot be dropped,
as shown by Example~2.6 in Bj\"orn--Bj\"orn--Lehrb\"ack~\cite{BBL}
where $X$ is the usual Cantor ternary set and $f$ is the Cantor
staircase function for which $f'(\rho) = 0 \le f(\rho)/\rho$
for a.e.\ $\rho >0$.

Example~\ref{ex-weighted-bow-tie} shows that 
the $1$-Poincar\'e assumption (for the last part)
 cannot be weakened, even under
global assumptions,
while
Example~\ref{ex-weighted-1D} shows that the reverse-doubling
assumption  cannot be dropped
even if $X$ supports a global $1$-Poincar\'e inequality.

\begin{proof}
\ref{g-annular} $\imp$ \ref{g-lip}
By the $1$-AD property at $x_0$  we have for $0 < \eps  < \rho$ that
\begin{equation} \label{eq-f}
    \frac{f(\rho)-f(\rho-\eps)}{\eps} 
    \simle \frac{\biggl(1-\displaystyle\frac{\rho-\eps}{\rho}\biggr) f(\rho)}{\eps}
    = \frac{f(\rho)}{\rho}.
\end{equation}
Since the right-hand side is locally bounded it follows
that $f$ is locally Lipschitz on $(0,\infty)$,
and thus that $f'(\rho)$ exists for a.e.\ $\rho>0$.
Moreover, by \eqref{eq-f} we see that $f'(\rho) \simle f(\rho)/\rho$
whenever $f'(\rho)$ exists.

\ref{g-lip} $\imp$ \ref{g-ac} This is trivial.

\ref{g-ac} $\imp$ \ref{g-annular}
Assume that $\rho f'(\rho)/f(\rho) \le M$ a.e.
We have
\begin{equation}  \label{eq-est-annular}
\frac{\mu(B_R\setm B_r)}{\mu(B_R)} = 1 -\frac{f(r)}{f(R)} 
= 1- \exp (h(r)-h(R)),
\end{equation}
where $h(\rho)=\log f(\rho)$ is also locally absolutely continuous with
$h'(\rho) = f'(\rho)/f(\rho) \le M/\rho$ for a.e.\ $\rho>0$.
It follows that
\[
h(r) - h(R) = - \int_r^R h'(\rho)\,d\rho \ge -M \int_r^R \frac{d\rho}{\rho}
= \log \Bigl( \frac{r}{R} \Bigr)^M.
\]
Inserting this into~\eqref{eq-est-annular} yields
\[
 1 -\frac{f(r)}{f(R)} 
\le 1 - \Bigl( \frac{r}{R} \Bigr)^M.
\]
Finally, Lemma~3.1 from 
Bj\"orn--Bj\"orn--Gill--Shanmugalingam~\cite{tree} shows that for $t\in[0,1]$,
\[
\min\{1,M\} t \le 1-(1-t)^M \le \max\{1,M\}t,
\]
and applying this with $t=1-r/R$ concludes the proof.

Thus we have shown that \ref{g-annular}--\ref{g-lip} are equivalent.
\medskip

Now assume that $\mu$ is 
 reverse-doubling at $x_0$, and that
$X$ supports a $1$-Poincar\'e inequality at $x_0$.

\ref{g-lip} $\imp$ \ref{g-extra}
Let $0< \rho < \frac{1}{3}\diam X$ and $0<\eps < \frac{1}{2} \rho$.
We have already shown that \ref{g-lip} $\imp$ \ref{g-annular},
so $\mu$ has the $1$-AD property at $x_0$, and 
in particular $\mu$ is doubling at $x_0$.
Thus Lemma~\ref{lem-reverse-decay-global} (with $q=1$) yields 
\[
    \frac{f(\rho)-f(\rho-\eps)}{\eps} 
    \simge \frac{\biggl(1-\displaystyle\frac{\rho-\eps}{\rho}\biggr) f(\rho)}{\eps}
    = \frac{f(\rho)}{\rho},
\]
showing that $f'(\rho) \simge f(\rho)/\rho$ whenever $f'(\rho)$ exists.

\ref{g-extra} $\imp$ \ref{g-lip}
If $X$ is unbounded this is trivial. 
So assume that $X$ is bounded.
If $\rho> \diam X$, then $f(\rho)=\mu(X)$ and $f'(\rho)=0$.
As $f$ is locally Lipschitz there is a constant $M$ such that
$f'(\rho) \le M$ for a.e.\ $\rho$ satisfying 
$\frac{1}{3}\diam X   < \rho <  \diam X$.
For such $\rho$ we have that
$f(\rho)/\rho \ge f(B_{\diam X /3})/{\diam X}$
and thus $\rho f'(\rho) \simle f(\rho)$ for a.e.\ $\rho > \frac{1}{3}\diam X$.
Together with \ref{g-extra} 
this yields \ref{g-lip}.
\end{proof}

In $\R^n$, the measure of a ball can be obtained by one-dimensional integration
of the surface measures of spheres.
To do the same in metric spaces 
we need the following lemma,
which is also useful for verifying 
the conditions in Theorem~\ref{thm-1-AD-global}.

\begin{lem}    \label{lem-AC}
Assume that $\mu$ is globally doubling
and that $X$ supports a global 
$q$-Poincar\'e inequality for some\/ $1\le q<\infty$. 
Assume, in addition, that there exists $a>0$
such that whenever\/ $0 < r < R \le 2r$ and $\de=R-r$,
for every $x\in B_{R}\setm B_{r}$ there exist 
$x'$ and $x''$ so that $B(x',a\de)\subset B(x,2\de)\cap B_{r}$ and 
$B(x'',a\de)\subset B(x,2\de)\setm B_{R}$.

Then the function $f(r):=\mu(B_r)$ is locally absolutely continuous
on\/ $(0,\infty)$.
\end{lem}

The values of the constants $a$ and $2$ (in $B(x,2\de)$) are not important,
and by covering $(0,\infty)$ we can even allow them to be different
in different parts.
Thus, we can replace the last assumption by the following
condition: for each $k \in \Z$ there exist $a_k>0$ and $b_k \ge 2$ such that
whenever $4^{k} < r < R \le 2r < 4^{k+2}$ and $\de=R-r$, 
for every $x\in B_{R}\setm B_{r}$ there exist 
$x'$ and $x''$ so that $B(x',a_k\de)\subset B(x,b_k\de)\cap B_{r}$ and 
$B(x'',a_k\de)\subset B(x,b_k\de)\setm B_{R}$.

\begin{proof}
It suffices to show that the measure $\nu$ defined on $(0,\infty)$ by 
\[
\nu(E)=\mu(\{x\in X: d(x_0,x)\in E\})
\] 
is absolutely continuous with respect to the Lebesgue measure.

Let $0 < r < R \le 2r$, $\de=R-r$ 
and $I=(r,R)$.
We start by showing that for all measurable $E\subset I$,
\begin{equation}   \label{eq-comp-E-I}
\biggl( \frac{|E|}{|I|} \biggr)^q \simle \frac{\nu(E)}{\nu(I)},
\end{equation}
where $q$ is the exponent from the assumed global $q$-Poincar\'e inequality.
To this end, set for $t>0$,
\[
u(x)=|E|-\int_0^{d(x_0,x)} \chi_E(\tau)\,d\tau, \quad x \in X,
\]
and note that $u=|E|$ in $B_{r}$, 
$u=0$ outside $B_{R}$
and $g_u(x)\le \chi_E(d(x_0,x))$ a.e.

Let the balls $B_j$, $B'_j$ and $B''_j$ be as in the proof of 
Theorem~\ref{thm-annular-PI}.
Hence, in the same way as in the proof of Theorem~\ref{thm-annular-PI}, 
we obtain
\[
(1-c)|E| \simle \de \biggl( \vint_{\la B_j} g_u^q \,d\mu \biggr)^{1/q} 
\le |I| \biggl( \frac{\mu(\{x\in {\la B_j}: d(x_0,x)\in E\})}
            {\mu(\la B_j\cap(B_{R}\setm B_{r}))} \biggr)^{1/q},
\]
or equivalently,
\[
|E|^q \, \mu(\la B_j\cap(B_{R}\setm B_{r}))
\simle |I|^q \mu(\{x\in {\la B_j}: d(x_0,x)\in E\}).
\]
Since the balls $\la B_j$ cover the annulus $B_{R} \setm B_{r}$ and have  bounded
overlap, summing over all $j$ gives~\eqref{eq-comp-E-I}.

Now assume for a contradiction that there exists $E\subset (r,R)$
such that $|E|=0$ and $\nu(E)>0$.
As $\nu$ is a Radon measure on $(0,\infty)$, 
the Lebesgue 
differentiation theorem holds with respect to $\nu$,
see Remark~1.13 in Heinonen~\cite{heinonen}.
Thus, there exists at least one $x\in E$
which is a Lebesgue point with respect to $\nu$ of the function
$1-\chi_E$.
Hence, for every $\eps>0$, there is an interval $I_\eps$ such that 
$x \in I_\eps \subset I$ and 
$\nu(I_\eps\setm E)<\eps \nu(I_\eps)$.
Applying~\eqref{eq-comp-E-I} to $I_\eps\setm E$ and $I_\eps$ in place of $E$
and $I$ gives
\[
1 = \biggl( \frac{|I_\eps \setm E|}{|I_\eps|} \biggr)^q 
\simle \frac{\nu(I_\eps\setm E)}{\nu(I_\eps)} < \eps,
\]
which is impossible.
Thus, the assumption that $\nu(E)>0$ was false and we have shown that $\nu$
is absolutely continuous with respect to the Lebesgue measure on every
interval $(r,R)$,
and hence on $(0,\infty)$.
\end{proof}

The following result is now a direct consequence of 
Theorem~\ref{thm-1-AD-global} 
and Lemma~\ref{lem-AC}.

\begin{cor}
Under the assumptions of Lemma~\ref{lem-AC}, $\mu$ has the\/ $1$-AD property at $x_0$
if and only if the function $f(r):=\mu(B_r)$ 
satisfies $\rho f'(\rho) \simle f(\rho)$ for a.e.\ $\rho>0$.
\end{cor}

\section{Counterexamples}
\label{sect-counterex}

In this section we provide a number of counterexamples
showing that most of our results are sharp.
The following example shows
the sharpness both of the Poincar\'e and AD assumptions in 
Theorem~\ref{thm-the-nice-case-intro}. 
It also shows sharpness of 
the Poincar\'e assumptions in several other results.

\begin{example} \label{ex-weighted-bow-tie}
(Weighted bow-tie)
Let 
\begin{equation} \label{eq-X}
  X=\bigl\{(x_1,\ldots,x_n)  : 
    x_2^2+\ldots+x_n^2  \le \tfrac{1}{4}x_1^2 \text{ and } -1 \le x_1 \le 2 \bigr\}
\end{equation}
as a subset of $\R^n$,  $n \ge 2$, and equip $X$ with the measure
$d\mu=|x|^\alp \, dx$, where $\alp > -n$.
(Additionally, we can make this example into a length space if 
we equip $X$ with the inner metric (see \cite[Definition~4.41]{BBbook}),
which only makes a difference when calculating distances
between the two sides of the origin.)
Note that the constant $2$ 
in the range of $x_1$ in \eqref{eq-X} above
 was chosen so that we can
have $R=1 \le \frac{1}{3} \diam X$ below
as required in Theorem~\ref{thm-the-nice-case-intro}.

If $q\ge 1$, then $X$ supports a global $q$-Poincar\'e inequality if and only if
$q > n+\alp$ or $q=1\ge n+\alp$, see Example~5.7 in \cite{BBbook}. Moreover,
$\mu$ is globally doubling.

Let $x_0=(-1,0,\ldots,0)$ and $\eta=\min\{1,n+\alp\}$.
Then for $0 < r < R <\diam X$ we have
\[ 
   \mu(B_R) \simeq R^n
   \quad \text{and} \quad
   \mu(B_R \setm B_r) 
   \simle \Bigl(1-\frac{r}{R}\Bigr)^\eta \mu(B_R),
\]
which shows that $\mu$ has the $\eta$-AD property at $x_0$.
One can check that this is the extreme case showing that $\mu$
has the  global $\eta$-AD property (and that $\eta$ is optimal).

If $0 < \de < \frac{1}{2}$, then 
\[
\mu(B_1 \setm B_{1-\de})
     \simeq  \int_0^\de \rho^{\alp} \rho^{n-1} \, d\rho
   \simeq \de^{n+\alp},
\]
which shows that the Poincar\'e assumption in Lemma~\ref{lem-reverse-decay-global}
cannot be weakened.
Moreover, by Lemma~\ref{lem-upper-simple}, 
\[ 
   \cp(B_{1-\de},B_1) \simle \frac{1}{\de^p}  \mu(B_1 \setm B_{1-\de})
   \simeq \de^{n+\alp-p},
\] 
which shows that \eqref{eq-thm-the-nice-case-intro} fails if $n+\alp >1$,
and thus we cannot replace the assumption of a global $1$-Poincar\'e inequality
in Theorem~\ref{thm-the-nice-case-intro} by a global $q$-Poincar\'e 
inequality for any fixed $q>1$.
Nor can the pointwise $q$-Poincar\'e inequality in 
Corollary~\ref{cor-Jana-2}
be replaced by assuming that $X$ supports a 
global $q'$-Poincar\'e inequality
for any fixed $q'>q$.

Conversely, if $0 < \de < \frac{1}{2}$ and $X$ supports 
a global \p-Poincar\'e inequality, i.e.\ if $p>n+\al$,
then a simple reflection argument and \cite[Proposition~10.8]{BBL}
imply that
\[
\cp(B_{1-\de}, B_1) \simge \cp(\{0\}, B(0,2\de) ) \simeq \de^{n+\alp-p}
\]
and hence
\[ 
   \cp(B_{1-\de},B_1) \simeq \de^{n+\alp-p}.
\] 
If $\eta=n+\alp <1$, then $X$ supports a global $1$-Poincar\'e inequality
and $\mu$ has the $\eta$-AD property at $x_0$, but
\eqref{eq-thm-the-nice-case-intro} fails.
Hence we cannot replace the $1$-AD assumption 
in Theorem~\ref{thm-the-nice-case-intro} by the $\eta$-AD property 
for any fixed $\eta <1$.
In fact, it is only the upper bound in \eqref{eq-thm-the-nice-case-intro}
that fails. The lower bound therein is still provided by 
Corollary~\ref{cor-Jana-2}.

Next, if $p=n+\alp>1$,
then $X$ supports a global $q$-Poincar\'e inequality for each $q>p$, 
but not a global \p-Poincar\'e inequality.
Moreover, by Example~5.7 in \cite{BBbook}, $\Cp(\{0\})=0$ 
and thus we can
test $\cp(B_{1-\de},B_1)$ with $u=\chi_{B_1}$ yielding
$\cp(B_{1-\de},B_1)=0$.
It also follows from Proposition~\ref{prop-connected}
that $X$ does not support a \p-Poincar\'e inequality at $x_0$.
Hence the \p-Poincar\'e assumption in
Proposition~\ref{prop-PI-lower-p=1} cannot be weakened if $p>1$.
Moreover, it also follows that it is not enough to assume that $X$ supports
global $q$-Poincar\'e inequalities for all $q>p$ in 
Theorems~\ref{thm-PI-and-AD} and \ref{thm-annular-PI} when $p>1$,
as well as for \eqref{eq-infty} in Proposition~\ref{prop-infty} to hold. 
As in this case we also have $\cp(B_{1},B_{1+\de})=0$, the
same is true for \eqref{eq-infty-outer} in Proposition~\ref{prop-infty}.

When $p=1<q$ we instead choose $n$ and $\alp$ so that $1 < n+ \alp < q$.
In particular, $X$ supports a global $q$-Poincar\'e inequality
in this case.
As above, $\cone(B_{1-\de},B_1)=0$ and
$X$ does not support a $1$-Poincar\'e inequality at $x_0$, showing
that the Poincar\'e assumption in 
Proposition~\ref{prop-PI-lower-p=1} is sharp also for $p=1$.
It also follows that when $p=1$ it is not enough 
to assume that $X$ supports
a global $q$-Poincar\'e inequality for some fixed $q>1$ in 
Theorem~\ref{thm-annular-PI}.

Let now, as in Theorem~\ref{thm-1-AD-global}, $f(r)=\mu(B_r)$.
Let $q>1$ and choose $n$ and $\alp$ 
so that $1<n+\alp <q$. 
Then $\mu$ has the global $1$-AD property and $X$ supports a 
global $q$-Poincar\'e
inequality.
For $\frac{1}{2} \le r < R \le 1$, with $R$ close to $r$, we see that
\[
   \mu(B_R \setm B_r) 
   \simeq (1-r)^\alp m(B_R\setm B_r)
   \simeq (1-r)^\alp (R-r) (1-r)^{n-1}
\]
where $m$ is the $n$-dimensional Lebesgue measure.
Hence 
\[ 
    rf'(r) = r \lim_{R\to r\limplus}\frac{\mu(B_R\setm R_r)}{R-r}
\simeq r (1-r)^{n+\alp-1} \not \simeq r 
\simeq r^n \simeq \mu(B_r)
\quad \text{when } \tfrac{1}{2} < r < 1.
\]
Thus condition \ref{g-extra} in Theorem~\ref{thm-1-AD-global}
fails,
which shows that it is not enough to assume that $X$ supports
a global $q$-Poincar\'e inequality for some fixed $q>1$ in 
(the last part of) Theorem~\ref{thm-1-AD-global}.
\end{example}

We do not have a counterexample to the conclusion of
Theorem~\ref{thm-the-nice-case}
which supports pointwise $q$-Poincar\'e inequalities at $x_0$ for all $q>1$.

\begin{example}\label{ex-starving-snake}
(This example was introduced by Tessera~\cite[p.\ 50]{Tessera} in
a different context. See also Routin~\cite[Section~6]{Routin}
for a more detailed discussion of this space.)
Let $X$ consist of the intervals
(with the natural embedding of $\R$ into $\R^2$)
$[0,1]$,
$[2^{k-1},2^k]$ for even positive $k$
and $[-2^{k},-2^{k-1}]$ for odd positive $k$,
and of the half-circles centred at the origin and of radius $2^k$ lying in the
upper half-plane for even nonnegative $k$ 
and in the lower half-plane for odd positive $k$. 
We equip $X$ with the Euclidean metric $d$ inherited from
$\R^2$ and the $1$-dimensional Hausdorff measure $\mu$. 
Then $X$ is Ahlfors $1$-regular
(see Proposition~6.1 in~\cite{Routin}).
Moreover, $X$ is bi-Lipschitz equivalent to the half-line $[0,\infty)\subset\R$, 
and hence  supports a global $1$-Poincar\'e inequality
(by \cite[Proposition~4.16]{BBbook}).

Let $x_0=0$, $k>0$ be an integer, $\de>0$ be small, 
$r=2^k-\de$ and  $R=2^k+\de$.
Then $\mu(B_R\setminus B_r)\simeq R\simeq \mu(B_R)$, which
shows that $\mu$ does not have the $\eta$-AD property at $x_0$
for any $\eta>0$.

Considering the function $u$ which is $1$ when $|x|<2^k$ and
$0$ when $|x|>2^k$, and decays linearly from $1$ to $0$ along the half-circle
of radius $2^k$, it is easy to see that for the above balls
$\cp(B_r,B_R)\simle R^{1-p} \simeq \mu(B_R)R^{-p}$.
Together with Proposition~\ref{prop-PI-lower-p=1} 
this shows that $\cp(B_r,B_R)\simeq \mu(B_R)R^{-p}$,
and thus the lower bound in Proposition~\ref{prop-PI-lower-p=1}
is sharp.
Moreover, the above estimate
shows that the geometric assumption in Theorem~\ref{thm-annular-PI}
and Corollary~\ref{cor-Jana-2} cannot be dropped, 
that the $\eta$-AD property cannot be replaced by global doubling
in Theorem~\ref{thm-PI-and-AD}, and also
that the assumption $\mu(\{y : d(y,x_0)=R\})=0$ cannot be dropped
for the limit \eqref{eq-infty-outer} in Proposition~\ref{prop-infty}
to hold, even if $X$ supports a global $1$-Poincar\'e inequality.
\end{example}

\begin{example} \label{ex-weighted-1D}
Let $w$ be a positive nonincreasing weight function on $X=[0,\infty)$,
$d\mu =w\, dx$ 
and $x_0=0$.
Assume that $\mu(B_1)<\infty$.
As $w$ is nonincreasing it is easy to see that $\mu$ is
doubling at $x_0$.
Let $f$ be an integrable function on $X$
with an upper gradient $g$,
and let $B=B(x,r) \subset X$ be a ball. Then either $B=(a,b)$ with $0 \le a<b$ or
$B=[a,b)$ with $a=0<b$. In either case we have
\begin{align*}
   \vint_{B} |f -f(a)| \, d\mu 
   & \le \frac{1}{\mu(B)} \int_a^b \int_a^t g(x) \, dx \, d\mu(t) 
    = \frac{1}{\mu(B)} \int_a^b \int_x^b  \, d\mu(t) g(x) \, dx \\
   & \le \frac{1}{\mu(B)} \int_a^b r w(x) g(x) \, dx 
    = r  \vint_{B}  g \, d\mu.
\end{align*}
It thus follows from Lemma~4.17 in \cite{BBbook} 
 that $X$ supports a global $1$-Poincar\'e inequality. 

Moreover, if $0 <\frac{1}{2}R  \le r <R$, then
\[
  \mu(B_R \setm B_r) 
   \le w(r) (R-r) 
   \le \frac{\mu(B_r)}{r} (R-r)
   \le 2 \Bigl(1-\frac{r}{R}\Bigr) \mu(B_R).
\]
On the other hand, if $0 < 2r <R$, then
\[
  \mu(B_R \setm B_r) 
  \le \mu(B_R)
   \le 2 \Bigl(1-\frac{r}{R}\Bigr) \mu(B_R).
\]
Hence $\mu$ has the $1$-AD property at $x_0$.

So far we have just assumed that $w$ is nonincreasing, but 
now assume that $w(x)=\min\{1,1/x\}$.
If $R>2$, then by Lemma~\ref{lem-upper-simple}, 
\[
   \cp(B_{R/2},B_{R}) 
   \simle \frac{\mu(B_R \setm B_{R/2})}{R^p}
   = \frac{\log 2}{R^p},
\]
while the right-hand sides (with $r=\frac{1}{2}R$) in
Theorem~\ref{thm-PI-and-AD}, 
Proposition~\ref{prop-PI-lower-p=1} and
Theorem~\ref{thm-the-nice-case}
are larger than this 
when $R$ is large enough,
since $\mu(B_R) \to \infty$, as $R \to \infty$.
In particular it follows that $\mu$ cannot be reverse-doubling at $x_0$
(which also follows directly from $\mu(B_{R}\setm B_{R/2})=\log 2$)
and that the reverse-doubling assumption in 
Theorem~\ref{thm-PI-and-AD}, 
Proposition~\ref{prop-PI-lower-p=1} and
Theorem~\ref{thm-the-nice-case}
cannot be dropped.

Moreover, as $\mu(B_R \setm B_{R/2})=\log 2$ and $\mu(B_R) \to \infty$
as $R \to \infty$, the inequality
in Lemma~\ref{lem-reverse-decay-global} fails in this case,
showing that the reverse-doubling assumption in
Lemma~\ref{lem-reverse-decay-global} cannot be dropped either.

Write 
$f(r)=\mu(B_r)$, as in Theorem~\ref{thm-1-AD-global}.
Then $f(r)=1+\log r$ for $r \ge 1$.
For $\rho >1$ we have $\rho f'(\rho)=1 \not \simeq f(\rho)$, 
so condition \ref{g-extra} in Theorem~\ref{thm-1-AD-global}
fails.
Thus the reverse-doubling assumption in 
(the last part of) Theorem~\ref{thm-1-AD-global}
cannot be dropped.

Finally, if we instead
let $w(r)=e^{-r}$, then $\mu(B_R)= 1- e^{-R}$. 
By Lemma~\ref{lem-upper-simple} we have for $R>1$ that
\[
   \cp(B_{R/2},B_R) \le \cp(B_{3R/4},B_R)
   \simle \frac{\mu(B_R \setm B_{3R/4})}{R^p} 
   \simle \frac{e^{-3R/4}}{R^p}. 
\]
As $\mu(B_R \setm B_{R/2}) \simeq e^{-R/2}$ for $R>1$,
this shows that the estimate in Theorem~\ref{thm-annular-PI}
fails in this case. Thus
the global doubling assumption therein cannot be replaced by assuming
that $\mu$ is doubling at $x_0$.
Note that the geometric assumption in Theorem~\ref{thm-annular-PI}
is satisfied in this case.
\end{example}

\begin{example} \label{ex-weighted-1D-inc}
Let this time $w$ be a positive nondecreasing weight function on $X=[0,\infty)$,
$d\mu =w\, dx$ 
and $x_0=0$.
As in Example~\ref{ex-weighted-1D}
we get that
$X$ supports a global $1$-Poincar\'e inequality
(estimate
using the right end point of the ball instead of the left end point).

Let $B$ be a ball with right end point $b$. Then 
$\mu(2B \setm B) \ge \mu(\{x \in 2B : x >b\}) \ge \frac{1}{2} \mu(B)$,
as $w$ is nondecreasing. Hence $\mu$ is globally reverse-doubling,
i.e.\ reverse-doubling at
every $x \in X$ with uniform constants.

Now let 
\[
  w(x)=\begin{cases}
    e^{-1/x}/x^2, & 0 \le x \le \tfrac{1}{2}, \\
    4e^{-2}, & x \ge \tfrac{1}{2},
    \end{cases}
\]
which is a continuous nondecreasing function
such that $\mu(B_R)=e^{-1/R}$ when $0< R < \tfrac{1}{2}$.
By Lemma~\ref{lem-upper-simple},
\[
   \cp(B_{R/2},B_R) 
   \le \cp(B_{R/2},B_{3R/4}) 
   \simle \frac{\mu(B_{3R/4})}{R^p}.
\]
As 
\[
   \frac{\mu(B_{3R/4})}{\mu(B_R)} = e^{-1/3R} \to 0,
   \quad \text{as } R \to 0\limplus,
\]
we see that the lower bound in Proposition~\ref{prop-PI-lower-p=1} fails and
that the doubling assumption cannot be dropped therein.
As $\mu(B_R) \simeq \mu(B_R \setm B_{R/2})$, this also shows that 
the estimate 
in Theorem~\ref{thm-annular-PI}
fails in this case. Thus
the global doubling assumption therein cannot be replaced by assuming
that $\mu$ is globally reverse-doubling.
Note that the geometric assumption in Theorem~\ref{thm-annular-PI}
is satisfied in this case.
\end{example}

In the last example $\mu(B_r) \not \simeq \mu(B_R)$
for $0<\tfrac{1}{2} R\le r < R$ and it is
natural to ask if it is possible to get the lower bound
$\mu(B_r)r^{-p}$ in Proposition~\ref{prop-PI-lower-p=1} 
without assuming that $\mu$ is doubling at $x_0$. 
At least for $p=1$ this is in fact possible.

\begin{prop}\label{prop-PI-lower-p=1-no-doubl}
Assume that $X$ supports a\/ $1$-Poincar\'e inequality at $x_0$
and that $\mu$ is  reverse-doubling at $x_0$ 
with dilation $\tau>1$.
Then
\begin{equation} \label{eq-PI-and-AD-p=1-no-doubl}
\cone(B_r,B_R)\simge \frac{\mu(B_r)}{r},
\quad \text{if\/ }0<\frac{R}{2}\le r<R \le \frac{\diam X}{2\tau}.
\end{equation}
\end{prop}

Example~\ref{ex-weighted-bow-tie} shows that
the $1$-Poincar\'e assumption cannot be weakened, even if
it is assumed globally.
Example~\ref{ex-weighted-1D} shows that the reverse-doubling
assumption cannot be dropped.

\begin{proof}
Let $u$ be admissible for $\cone(B_r,B_R)$.
As in the proof of Theorem~\ref{thm-PI-and-AD}, we get that
\begin{align*}
1 & \simle\vint_{B_{r}}|u-u_{B_{\tau R}}|\,d\mu 
   \le \frac{\mu(B_{\tau R})}{\mu(B_r)}\vint_{B_{\tau R}}|u-u_{B_{\tau R}}|\,d\mu \\
&  \simle R\, \frac{\mu(B_{\tau R})}{\mu(B_r)} \vint_{B_{\tau\lambda  R}} g_{u}\,d\mu 
 \simle \frac{r}{\mu(B_r)} \int_{B_{R}\setminus B_r} g_{u} \,d\mu,
\end{align*} 
and \eqref{eq-PI-and-AD-p=1-no-doubl} follows after taking infimum over all
admissible $u$.
\end{proof}

\end{document}